\documentclass[10pt,reqno]{amsart}
\usepackage{bbm}
\usepackage{amsmath}
\usepackage{mathrsfs}
\usepackage{amsfonts} 
\usepackage{graphicx,latexsym,bm,amsmath,amssymb,verbatim,multicol,lscape}
\newtheorem{thm}{Theorem} [section]
\newtheorem{cor}[thm]{Corollary}
\newtheorem{lem}[thm]{Lemma}
\newtheorem{prop}[thm]{Proposition}
\theoremstyle{definition}

\theoremstyle{remark}

\numberwithin{equation}{section}

\begin{document}
\title{Permutational behavior of reversed Dickson polynomials over finite fields}%
\author[Kaimin Cheng]
{Kaimin Cheng\\
Department of Mathematics, Sichuan University Jinjiang College,
Pengshan 620860, P.R. China}
\thanks{The research was supported partially by the General Project of Department
of Education of Sichuan Province \# 15ZB0434. \\
Email: ckm20@126.com.}

\keywords{Permutation polynomial, Reversed Dickson polynomial
of the fourth kind, Finite field, Generating function}
\subjclass[2000]{Primary 11T06, 11T55, 11C08}
\date{\today}%
\begin{abstract}
In this paper, we use the method developed previously by Hong,
Qin and Zhao to obtain several results on the permutational
behavior of the reversed Dickson polynomial $D_{n,k}(1,x)$
of the $(k+1)$-th kind over the finite field ${\mathbb F}_{q}$.
Particularly, we present the explicit evaluation of the
first moment $\sum_{a\in {\mathbb F}_{q}}D_{n,k}(1,a)$.
Our results extend the known results from the case
$0\le k\le 3$ to the general $k\ge 0$ case.
\end{abstract}

\maketitle

\section{Introduction}
Let ${\mathbb F}_{q}$ be the finite field of characteristic $p$ with $q$ elements.
Associated to any integer $n\ge 0$ and a parameter $a\in {\mathbb F}_{q}$,
the $n$-th \textit{Dickson polynomials of the first kind and of the second kind},
denoted by $D_n(x,a)$ and $E_n(x,a)$, are defined for $n\ge 1$ by
$$D_n(x,a):=\sum_{i=0}^{[\frac{n}{2} ]}\frac{n}{n-i}\binom{n-i}{i}(-a)^ix^{n-2i}$$
and
$$E_n(x,a):=\sum_{i=0}^{[\frac{n}{2} ]}\binom{n-i}{i}(-a)^ix^{n-2i},$$
respectively, and $D_0(x,a):=2, E_0(x,a):=1$, where $[\frac{n}{2}]$
means the largest integer no more than $\frac{n}{2}$. In 2012, Wang and Yucas
\cite{[WY]} further defined the \textit{$n$-th Dickson polynomial
of the $(k+1)$-th kind $D_{n,k}(x,a)\in {\mathbb F}_{q}[x]$} for $n\ge 1$ by
$$D_{n,k}(x,a):=\sum_{i=0}^{[\frac{n}{2} ]}\frac{n-ki}{n-i}\binom{n-i}{i}(-a)^ix^{n-2i}$$
and $D_{0,k}(x,a):=2-k$.

Hou, Mullen, Sellers and Yucas \cite{[HMSY]} introduced the definition of
\textit{the reversed Dickson polynomial of the first kind},
denoted by $D_n(a,x)$, as follows

$$D_n(a,x):=\sum_{i=0}^{[\frac{n}{2} ]}\frac{n}{n-i}\binom{n-i}{i}(-x)^ia^{n-2i}$$
if $n\ge 1$ and $D_0(a, x)=2$.
To extend the definition of reversed Dickson polynomials, Wang and Yucas \cite{[WY]}
defined \textit{the $n$-th reversed Dickson polynomial of
$(k+1)$-th kind} $D_{n,k}(a,x)\in {\mathbb F}_{q}[x]$, which is defined for $n\ge 1$ by
\begin{equation}\label{c0}
D_{n,k}(a,x):=\sum_{i=0}^{[\frac{n}{2} ]}\frac{n-ki}{n-i}\binom{n-i}{i}(-x)^ia^{n-2i}
\end{equation}
and $D_{0,k}(a,x)=2-k$.

It is well known that $D_n(x,0)$ is a permutation polynomial of ${\mathbb F}_{q}$
if and only if $\gcd(n,q-1)=1$, and if $a\ne 0$, then $D_n(x,a)$ induces a permutation
of ${\mathbb F}_{q}$ if and only if $\gcd(n,q^2-1)=1$. Besides, there are lots of
published results on permutational properties of Dickson polynomial $E_n(x,a)$
of the second kind (see, for example, \cite{[Coh]}). In \cite{[WY]},
Wang and Yucas investigated the permutational properties of Dickson polynomial
$D_{n,2}(x,1)$ of the third kind. They got some necessary
conditions for $D_{n,2}(x,1)$ to be a permutation polynomial of ${\mathbb F}_{q}$.

Hou, Mullen, Sellers and Yucas \cite{[HMSY]} considered the permutational
behavior of reversed Dickson polynomial $D_{n}(a,x)$ of the first kind.
Actually, they showed that $D_n(a,x)$ is closely related to almost perfect
nonlinear functions, and obtained some families of permutation polynomials
from the revered Dickson polynomials of the first kind. In \cite{[HL]},
Hou and Ly found several necessary conditions for the revered Dickson
Polynomials $D_n(1,x)$ of the first kind to be a permutation polynomial.
Recently, Hong, Qin and Zhao \cite{[HQZ]} studied the revered Dickson
polynomial $E_n(a,x)$ of the second kind that is defined for $n\ge 1$ by
$$E_{n}(a,x):=\sum_{i=0}^{[\frac{n}{2} ]}\binom{n-i}{i}(-x)^ia^{n-2i}$$
and $E_{0}(a,x)=1$. In fact, they gave some necessary conditions for
the revered Dickson polynomial $E_n(1,x)$ of the second kind to be
a permutation polynomial of ${\mathbb F}_{q}$. Regarding the revered Dickson
polynomial $D_{n, 2}(a,x)\in {\mathbb F}_{q}[x]$ of the third kind,
from its definition one can derive that
\begin{equation}\label{c1}
D_{n,2}(a,x)=E_{n-1}(a,x)
\end{equation}
for each $x\in {\mathbb F}_{q}$. Using (\ref{c1}), one can deduce
immediately from \cite{[HQZ]} the similar results on the permutational
behavior of the reversed Dickson polynomial $D_{n,2}(a,x)$ of the third kind.
On the other hand, by using the method presented by Hong, Qin and Zhao
in \cite{[HQZ]}, Cheng, Hong and Qin \cite{[CHQ]} obtained the results
on the permutational behavior of the reversed Dickson polynomial
$D_{n,3}(a,x)$ of the fourth kind.

In this paper, our main goal is to continue to use the method developed by
Hong, Qin and Zhao in \cite{[HQZ]} to investigate the reversed
Dickson polynomial $D_{n,k}(a,x)$ of the $(k+1)$-th kind which
is defined by (\ref{c0}) if $n\ge 1$ and $D_{0,k}(a,x):=2-k$.
For $a\ne 0$, we write $x=y(a-y)$ with an indeterminate $y\ne \frac{a}{2}$.
Then one can rewrite $D_{n,k}(a,x)$ as
\begin{equation}\label{c3}
D_{n,k}(a,x)=\frac{\big((k-1)a-(k-2)y\big)y^n-\big(a+(k-2)y\big)(a-y)^n}{2y-a}.
\end{equation}
We have
\begin{equation}\label{c3'}
D_{n, k}\Big(a, \frac{a^2}{4}\Big)=\frac{(kn-k+2)a^n}{2^n}.
\end{equation}
In fact, (\ref{c3}) and (\ref{c3'}) follow from Theorem 2.2 (i)
and Theorem 2.4 (i) below. It is easy to see that if ${\rm char}({\mathbb F}_{q})=2$,
then $D_{n,k}(a,x)=E_n(a,x)$ if $k$ is odd and $D_{n,k}(a,x)=D_n(a,x)$ if $k$ is even.
We also find that $D_{n,k}(a, x)=D_{n,k+p}(a, x)$, so we can restrict $p>k$.
Thus we always assume $p={\rm char}({\mathbb F}_{q})>3$ in what follows.

The paper is organized as follows. First in section 2, we study the
properties of the reversed Dickson polynomial $D_{n,k}(a,x)$ of the
fourth kind. Subsequently, in Section 3, we prove a necessary condition
for the reversed Dickson polynomial $D_{n,k}(1,x)$ of the $k+1$-th kind
to be a permutation polynomial of ${\mathbb F}_{q}$ and then
introduce an auxiliary polynomial to present a characterization
for $D_{n,k}(1,x)$ to be a permutation of ${\mathbb F}_q$.
From the Hermite criterion \cite{[LN]} one knows that a function
$f: {\mathbb F}_{q}\rightarrow {\mathbb F_{q}}$
is a permutation polynomial of ${\mathbb F_{q}}$
if and only if the $i$-th moment
\begin{align*}\nonumber
\sum_{a\in {\mathbb F_{q}}}f(a)^i=
{\left\{\begin{array}{rl}
0,& {\rm if} \ 0\le i\le q-2,\\
-1,& {\rm if} \ i=q-1.
\end{array}
\right.}
\end{align*}
Thus to understand well the permutational behavior of
the reversed Dickson polynomial $D_{n,k}(1,x)$
of the fourth kind, we would like to know if the $i$-th
moment $\sum_{a\in {\mathbb F}_q}D_{n,k}(1,a)^i$
is computable. We are able to treat with this sum when $i =1$.
The final section is devoted to the computation of the first moment
$\sum_{a\in {\mathbb F_{q}}}D_{n,k}(1,a)$.

\section{Reversed Dickson polynomials of the $k+1$-th kind}

In this section, we study the properties of the reversed Dickson polynomials
$D_{n,k+1}(a,x)$ of the fourth kind. Clearly, if $a=0$, then
\begin{align*}\nonumber
D_{n,k+1}(0,x)={\left\{
  \begin{array}{rl}
0,& {\rm if} \ n \ {\rm is \ odd},\\
(-1)^{\frac{n}{2}+1}(k-2)x^{\frac{n}{2}},& {\rm if} \ n \ {\rm is \ even}.
\end{array}
\right.}
\end{align*}
Therefore, $D_{n,k+1}(0,x)$ is a PP (permutation polynomial) of ${\mathbb F_{q}}$
if and only if $n$ is an even integer with $\gcd(\frac{n}{2},q-1)=1$.
In what follows, we always let $a\in {\mathbb F_{q}^*}$.
First, we give a basic fact as follows.
\begin{lem}
\cite{[LN]} Let $f(x)\in {\mathbb F_{q}}[x]$. Then $f(x)$
is a PP of ${\mathbb F_{q}}$ if and only if
$cf(dx)$ is a PP of ${\mathbb F_{q}}$ for any
given $c, d\in {\mathbb F_{q}^*}$.
\end{lem}
Then we can deduce the following result.
\begin{thm}
Let $a, b\in {\mathbb F_{q}^*}$. Then the following are true.

{\rm (i)}. One has $D_{n,k}(a,x)=\frac{a^n}{b^n}D_{n,k}(b,\frac{b^2}{a^2}x)$.

{\rm (ii)}. We have that $D_{n,k}(a,x)$ is a PP of ${\mathbb F_{q}}$
if and only if $D_{n,k}(1,x)$ is a PP of ${\mathbb F_{q}}$.
\end{thm}

\begin{proof}
{\rm (i)}. By the definition of $D_{n,k}(a,x)$, we have
\begin{align*}
&\frac{a^n}{b^n}D_{n,k}\Big(b,\frac{b^2}{a^2}x\Big)\\
=&\frac{a^n}{b^n}
\sum_{i=0}^{[\frac{n}{2} ]}\frac{n-ki}{n-i}\binom{n-i}{i}
(-1)^ib^{n-2i}\frac{b^{2i}}{a^{2i}}x^{i}\\
=&\sum_{i=0}^{[\frac{n}{2}]}\frac{n-ki}{n-i}\binom{n-i}{i}(-1)^ia^{n-2i}x^{i}\\
=&D_{n,k}(a,x)
\end{align*}
as required. Part (i) is proved.

(ii). Taking $b=1$ in part (i), we have
$$D_{n,k}(a,x)=a^nD_{n,k}\Big(1,\frac{x}{a^2}\Big).$$
It then follows from Lemma 2.1 that $D_{n,k}(a,x)$ is a PP
of ${\mathbb F_{q}}$ if and only if
$D_{n,k}(1,x)$ is a PP of ${\mathbb F_{q}}$.
This completes the proof of part (ii). So Theorem 2.2 is proved.
\end{proof}

Theorem 2.2 tells us that to study the permutational
behavior of $D_{n,k}(a,x)$ over ${\mathbb F_{q}}$,
one only needs to consider that of $D_{n,k}(1,x)$.
In the following, we supply several basic properties
on the reversed Dickson polynomial $D_{n,k}(1,x)$
of the fourth kind. The following result is given
in \cite{[HQZ]} and \cite{[HMSY]} without proof.
For its proof, one can see \cite{[CHQ]}.

\begin{lem} \cite{[HQZ]} \cite{[HMSY]}
Let $n\ge 0$ be an integer. Then we have
$$D_n(1,x(1-x))=x^n+(1-x)^n$$
and
$$E_{n}(1,x(1-x))=\frac{x^{n+1}-(1-x)^{n+1}}{2x-1}$$
if $x\ne \frac{1}{2}.$
\end{lem}

\begin{thm}
Each of the following is true.

{\rm (i)}. For any integer $n\ge 0$, we have
$$D_{n,k}(1,\frac{1}{4})=\frac{kn-k+2}{2^n}$$
and
$$D_{n,k}(1,x)=\frac{\big(k-1-(k-2)y\big)y^n-\big(1+(k-2)y\big)(1-y)^n}{2y-1}$$
if $x=y(1-y)\ne \frac{1}{4}$.

{\rm (ii)}. If $n_1$ and $n_2$ are positive integers such that
$n_1\equiv n_2\pmod{q^2-1}$, then one has
$D_{n_1,k}(1,x_0)=D_{n_2,k}(1,x_0)$ for any
$x_0\in {\mathbb F_{q}\setminus\{\frac{1}{4}\}}$.
\end{thm}
\begin{proof}
(i). First of all, it is easy to see that
$D_{0,k}\big(1,\frac{1}{4}\big)=2-k=\frac{k\times 0-k+2}{2^0}$
and
$D_{1,k}\big(1,\frac{1}{4}\big)=1=\frac{k\times 1-k+2}{2^1}$.
the first identity is true for the cases that $n=0$ and 1.
Now let $n\ge 2$. Then one has
\begin{align*}
D_{n,k}\Big(1,\frac{1}{4}\Big)&=\sum_{i=0}^{[\frac{n}{2}]}
\frac{n-ki}{n-i}\binom{n-i}{i}\Big(-\frac{1}{4}\Big)^{i}\\
&=\sum_{i=0}^{[\frac{n}{2} ]}\frac{n-(k-1)i}{n-i}\binom{n-i}{i}
\Big(-\frac{1}{4}\Big)^{i}+\sum_{i=0}^{[\frac{n}{2} ]}
\frac{-i}{n-i}\binom{n-i}{i}\Big(-\frac{1}{4}\Big)^{i}\\
&=D_{n,k-1}\Big(1,\frac{1}{4}\Big)+\frac{1}{4}\sum_{i=0}^{[\frac{n}{2} ]-1}
\binom{n-2-i}{i}\Big(-\frac{1}{4}\Big)^{i}\\
&=D_{n,k-1}\Big(1,\frac{1}{4}\Big)+\frac{1}{4}E_{n-2}\Big(1,\frac{1}{4}\Big),
\end{align*}
which follows from Theorem 2.2 (1) in \cite{[HQZ]} that
\begin{align*}
D_{n,k}\Big(1,\frac{1}{4}\Big)&=D_{n,1}\Big(1,\frac{1}{4}\Big)+(k-1)
\frac{1}{4}E_{n-2}\Big(1,\frac{1}{4}\Big)\\
&=\frac{n+1}{2^n}+\frac{(k-1)n-(k-1)}{2^n}\\
&=\frac{kn-k+2}{2^n}
\end{align*}
as desired. So the first identity is proved.

Now we turn our attention to the second identity.
Let $x\ne \frac{1}{4}$, then there exists $y\in {\mathbb F_{q^2}\setminus\{\frac{1}{2}\}}$
such that $x=y(1-y)$ . So by the definition of
the $n$-th reversed Dickson polynomial
of the $k+1$-th kind, one has
\begin{align}
D_{n,k}(1,y(1-y))&=\sum_{i=0}^{[\frac{n}{2}]}
\frac{n-ki}{n-i}\binom{n-i}{i}(-y(1-y))^{i}\nonumber\\
&=\sum_{i=0}^{[\frac{n}{2}]}
\frac{k(n-i)-2n}{n-i}\binom{n-i}{i}(-y(1-y))^{i}\nonumber\\
&=k\sum_{i=0}^{[\frac{n}{2}]}
\binom{n-i}{i}(-y(1-y))^{i}-(k-1)\sum_{i=0}^{[\frac{n}{2}]}
\frac{n}{n-i}\binom{n-i}{i}(-y(1-y))^{i}\nonumber\\
&=kE_n(1,y(1-y))-(k-1)D_n(1,y(1-y))\label{ch1}.
\end{align}
But Lemma 2.3 gives us that
\begin{align}
D_n(1,y(1-y))=y^n+(1-y)^n\label{ch3}
\end{align}
and
\begin{align}
E_n(1,y(1-y))=\sum_{i=0}^{n}y^{n-i}(1-y)^i
=\frac{x^{n+1}-(1-x)^{n+1}}{2x-1}   \label{ch4}.
\end{align}
Thus it follows from (\ref{ch1}) to (\ref{ch4}) that
\begin{align*}
D_{n,k}(1, x)=&D_{n,k}(1,y(1-y))\\
=&kE_n(1,y(1-y))-(k-1)D_n(1,y(1-y))\\
=&\frac{ky^{n+1}-k(1-y)^{n+1}}{2y-1}-(k-1)\big(y^n+(1-y)^n\big)\\
=&\frac{\big(k-1-(k-2)y\big)y^n-\big(1+(k-2)y\big)(1-y)^n}{2y-1}
\end{align*}
as required. So the second identity holds. Part (i) is proved.

(ii). For each $x_0\in {\mathbb F_{q}\setminus\{\frac{1}{4}\}}$,
one can choose an element $y_0\in {\mathbb F_{q^2}\setminus\{\frac{1}{2}\}}$
such that $x_0=y_0(1-y_0)$. Since $n_1\equiv n_2\pmod{q^2-1}$,
one has $y_0^{n_1}=y_0^{n_2}$ and $(1-y_0)^{n_1}=(1-y_0)^{n_2}$.
It then follows from part (i) that
\begin{align*}
D_{n_1,k}(1,x_0)&=D_{n_1,k}(1,y_0(1-y_0))\\
&=\frac{\big(k-1-(k-2)y_0\big)y_0^{n_1}-\big(1+(k-2)y_0\big)(1-y_0)^{n_1}}{2y_0-1}\\
&=\frac{\big(k-1-(k-2)y_0\big)y_0^{n_2}-\big(1+(k-2)y_0\big)(1-y_0)^{n_2}}{2y_0-1}\\
&=D_{n_2,k}(1,x_0)
\end{align*}
as desired. This ends the proof of Theorem 2.4.
\end{proof}

Evidently, by Theorem 2.2 (i) and Theorem 2.4 (i) one can derive
that (1.3) and (1.4) are true.

\begin{prop}
Let $n\ge 2$ be an integer. Then the recursion
$$D_{n,k}(1,x)=D_{n-1,k}(1,x)-xD_{n-2,k}(1,x)$$
holds for any $x\in {\mathbb F_{q}}$.
\end{prop}
\begin{proof} We consider the following two cases.

{\sc Case 1.} $x\ne \frac{1}{4}$. For this case, one may let $x=y(1-y)$
with $y\in {\mathbb F_{q^2}}\setminus \{ \frac{1}{2}\}$.
Then by Theorem 2.4 (i), we have
\begin{align*}
&\ \ \ \ \ D_{n-1,k}(1,x)-xD_{n-2,k}(1,x)\\
&=D_{n-1,k}(1,y(1-y))-y(1-y)D_{n-2,k}(1,y(1-y))\\
&=\frac{\big(k-1-(k-2)y\big)y^{n-1}-\big(1+(k-2)y\big)(1-y)^{n-1}}{2y-1}\\
&\ \ \ \ -y(1-y)\frac{\big(k-1-(k-2)y\big)y^{n-2}-\big(1+(k-2)y\big)(1-y)^{n-2}}{2y-1}\\
&=\frac{\big(k-1-(k-2)y\big)y^{n}-\big(1+(k-2)y\big)(1-y)^{n}}{2y-1}\\
&=D_{n,k}(1,x)
\end{align*}
as required.

{\sc Case 2.} $x=\frac{1}{4}$. Then by Theorem 2.4 (i), we have
\begin{align*}
&D_{n-1,k}\Big(1,\frac{1}{4}\Big)-\frac{1}{4}D_{n-2,k}\Big(1,\frac{1}{4}\Big)\\
=&\frac{k(n-1)-k+2}{2^{n-1}}-\frac{1}{4}\frac{k(n-2)-k+2}{2^{n-2}}\\
=&\frac{kn-k+2}{2^{n}}\\
=&D_{n,k}\Big(1,\frac{1}{4}\Big).
\end{align*}

This concludes the proof of Proposition 2.5.
\end{proof}

By Proposition 2.5, we can obtain the generating function of
the reversed Dickson polynomial $D_{n,k}(1,x)$ of the $k+1$-th kind as follows.
\begin{prop}
The generating function of $D_{n,k}(1,x)$ is given by
$$\sum_{n=0}^{\infty}D_{n,k}(1,x)t^n=\frac{(k-1)t-k+2}{1-t+xt^2}.$$
\end{prop}
\begin{proof}
By the recursion presented in Proposition 2.5, we have
\begin{align*}
&(1-t+xt^2)\sum_{n=0}^{\infty}D_{n,k}(1,x)t^n\\
=&\sum_{n=0}^{\infty}D_{n,k}(1,x)t^n-\sum_{n=0}
^{\infty}D_{n,k}(1,x)t^{n+1}+x\sum_{n=0}^{\infty}D_{n,k}(1,x)t^{n+2}\\
=&(k-1)t-k+2+\sum_{n=0}^{\infty}\big(D_{n+2,k}
(1,x)-D_{n+1,k}(1,x)+xD_{n,k}(1,x)\big)t^{n+2}\\
=&(k-1)t-k+2.
\end{align*}
Thus the desired result follows immediately.
\end{proof}

Now we can use Theorem 2.4 to present an explicit formula for $D_{n, k}(1, x)$
when $n$ is a power of the characteristic $p$. Then we show that $D_{n, k}(1, x)$
is not a PP of ${\mathbb F}_q$ in this case.

\begin{prop}
Let $p={\rm char}({\mathbb F_{q}})>3$ and $s$ be a positive integer. Then
$$2^{p^s}D_{p^s, k}(1,x)+k-2=k(1-4x)^{\frac{p^s-1}{2}}.$$
\end{prop}
\begin{proof}
We consider the following two cases.

{\sc Case 1.} $x\ne \frac{1}{4}$. For this case, putting $x=y(1-y)$ in Theorem 2.4 (i) gives us that
\begin{align*}
D_{p^s,k}(1,x)=&D_{p^s,k}(1,y(1-y))\\
=&\frac{\big(k-1-(k-2)y\big)y^{p^s}-\big(1+(k-2)y\big)(1-y)^{p^s}}{2y-1}\\
=&\frac{\frac{k+(2-k)u}{2}\big(\frac{u+1}{2}\big)^{p^s}
-\frac{k+(k-2)u}{2}\big(\frac{1-u}{2}\big)^{p^s}}{u}\\
=&\frac{1}{2^{p^s+1}u}\Big((k+(2-k)u)(u+1)^{p^s}-(k+(k-2)u)(1-u)^{p^s}\Big)\\
=&\frac{1}{2^{p^s}}(ku^{p^s-1}-k+2),
\end{align*}
where $u=2y-1$. So we obtain that
\begin{align*}
&2^{p^s}D_{p^s,k}(1,x)\\
=&k(u^2)^{\frac{p^s-1}{2}}-k+2\\
=&k\big((2y-1)^2\big)^{\frac{p^s-1}{2}}-k+2,
\end{align*}
which infers that
$$2^{p^s}D_{p^s,k}(1,x)+k-2=k(1-4x)^{\frac{p^s-1}{2}}$$
as desired.

{\sc Case 2.} $x=\frac{1}{4}$. By Theorem 2.4 (i), one has
$$2^{p^s}D_{p^s,k}\big(1,\frac{1}{4}\big)+k-2=2^{p^s}\frac{kp^s-k+2}{2^{p^s}}+k-2=0
=k(1-4\times\frac{1}{4})^{\frac{p^s-1}{2}}$$
as required.
So Proposition 2.7 is proved.
\end{proof}

It is well known that every linear polynomial over
${\mathbb F_{q}}$ is a PP of ${\mathbb F_{q}}$
and that the monomial $x^n$ is a PP of
${\mathbb F_{q}}$ if and only if $\gcd(n,q-1)=1$. Then by
Proposition 2.7, we have the following result.
\begin{cor}
Let $p>3$ be a prime and $q=p^e$. Let $e$ and $s$ be positive integers with
$s\le e$. Then $D_{p^s,k}(1,x)$ is not a PP of ${\mathbb F_{q}}$.
\end{cor}
\begin{proof}
By Proposition 2.7, we know that $D_{p^s,k}(1,x)$
is a PP of ${\mathbb F_{q}}$ if and only if
$$(1-4x)^{\frac{p^s-1}{2}}$$
is a PP of ${\mathbb F_{q}}$ which is equivalent to
$$\gcd\Big(\frac{p^s-1}{2},q-1\Big)=1.$$
The latter one is impossible since
$\frac{p-1}{2}|\gcd\big(\frac{p^s-1}{2},q-1\big)$ implies that
$$\gcd\Big(\frac{p^s-1}{2},q-1\Big)\ge\frac{p-1}{2}>1.$$
Thus $D_{p^s,k}(1,x)$ is not a PP of ${\mathbb F_{q}}$.
\end{proof}

\begin{prop}
Let $p={\rm char}({\mathbb F_{q}})>3$ and $s$ and $l$ be integers
such that $0<s<\ell$. Then
$$D_{p^s+p^{\ell}, k}(1,x)=\frac{k}{4}\big((1-4x)^{\frac{p^s-1}{2}}+
(1-4x)^{\frac{p^{\ell}-1}{2}}\big)-\frac{k-2}{4}\big(1
+(1-4x)^{\frac{p^{s}+p^{\ell}}{2}}\big).$$
\end{prop}
\begin{proof}
We consider the following two cases.

{\sc Case 1.} $x\ne \frac{1}{4}$. For this case, putting $x=y(1-y)$ in Theorem 2.4 (i) gives us that
\begin{align*}
D_{p^s+p^{\ell},k}(1,x)=&D_{p^s+p^{\ell},k}(1,y(1-y))\\
=&\frac{\big(k-1-(k-2)y\big)y^{p^s+p^{\ell}}-\big(1+(k-2)y\big)(1-y)^{p^s+p^{\ell}}}{2y-1}\\
=&\frac{\frac{k+(2-k)u}{2}\big(\frac{u+1}{2}\big)^{p^s+p^{\ell}}
-\frac{k+(k-2)u}{2}\big(\frac{1-u}{2}\big)^{p^s+p^{\ell}}}{u}\\
=&\frac{k}{4}(u^{p^s-1}+u^{p^{\ell}-1})-\frac{k-2}{4}(1+u^{p^s+p^{\ell}})\\
=&\frac{k}{4}\big((u^2)^{\frac{p^s-1}{2}}+(u^2)^{\frac{p^{\ell}-1}{2}}\big)-\frac{k-2}{4}(1+(u^2)^{\frac{p^s+p^{\ell}}{2}}),
\end{align*}
where $u=2y-1$ and $u^2=1-4x$. So we obtain that
$$D_{n,k}(1,x)=\frac{k}{4}\big((1-4x)^{\frac{p^s-1}{2}}+
(1-4x)^{\frac{p^{\ell}-1}{2}}\big)-\frac{k-2}{4}\big(1
+(1-4x)^{\frac{p^{s}+p^{\ell}}{2}}\big)$$
as desired.

{\sc Case 2.} $x=\frac{1}{4}$. By Theorem 2.4 (i), one has
$$D_{p^s+p^{\ell},k}\big(1,\frac{1}{4}\big)=\frac{k(p^s+p^{\ell})-k+2}{2^{p^s+p^{\ell}}}=\frac{-k+2}{4}.$$
Besides,
$$\frac{k}{4}\big((1-4\times \frac{1}{4})^{\frac{p^s-1}{2}}+
(1-4\times \frac{1}{4})^{\frac{p^{\ell}-1}{2}}\big)-\frac{k-2}{4}\big(1
+(1-4\times \frac{1}{4})^{\frac{p^{s}+p^{\ell}}{2}}\big)=\frac{-k+2}{4}.$$
Thus the required result follows.
So Proposition 2.9 is proved.
\end{proof}

\begin{lem}\cite{[HMSY]}
Let $x\in{\mathbb F_{q^2}}$. Then $x(1-x)\in{\mathbb F_{q}}$
if and only if $x^q=x$ or $x^q=1-x$.
\end{lem}

Let $V$ be defined by
$$V:=\{x\in {\mathbb F_{q^2}}: x^q=1-x\}.$$
Clearly, ${\mathbb F_{q}}\cap V=\{\frac{1}{2}\}$. Then we obtain a characterization for
$D_{n,k}(1,x)$ to be a PP of ${\mathbb F_{q}}$ as follows.
\begin{thm}
Let $q=p^e$ with $p>3$ being a prime and $e$ being a positive integer. Let
$$f:y\mapsto\frac{\big(k-1-(k-2)y\big)y^{n}-\big(1+(k-2)y\big)(1-y)^{n}}{2y-1}$$
be a mapping on $({\mathbb F_{q}}\cup V)\setminus \{\frac{1}{2}\}$. Then $D_{n,k}(1,x)$ is
a PP of ${\mathbb F_{q}}$ if and only if $f$ is $2$-to-$1$ and $f(y)\ne\frac{kn-k+2}{2^n}$
for any $y\in({\mathbb F_{q}}\cup V)\setminus \{\frac{1}{2}\}$.
\end{thm}
\begin{proof}
First, we show the sufficiency part. Let $f$ be $2$-to-$1$ and $f(y)\ne\frac{kn-k+2}{2^n}$
for any $y\in({\mathbb F_{q}}\cup V)\setminus \{\frac{1}{2}\}$.
Let $D_{n,k}(1,x_1)=D_{n,k}(1,x_2)$ for $x_1,x_2\in{\mathbb F_{q}}$. To show that
$D_{n, k}(1, x)$ is a PP of ${\mathbb F_{q}}$, it suffices to show that $x_1=x_2$
that will be done in what follows.

First of all, one can find $y_1,y_2\in{\mathbb F_{q^2}}$ satisfying
$x_1=y_1(1-y_1)$ and $x_2=y_2(1-y_2)$.
By Lemma 2.10, we know that $y_1,y_2\in{\mathbb F_{q}}\cup V$.
We divide the proof into the following two cases.

{\sc Case 1.} At least one of $x_1$ and $x_2$ is equal to $\frac{1}{4}$.
Without loss of any generality, we may let $x_1=\frac{1}{4}$.
So by Theorem 2.4 (i), one derives that
\begin{equation}\label{c4}
D_{n,k}(1,x_2)=D_{n,k}(1,x_1)=D_{n,3}\Big(1,\frac{1}{4}\Big)=\frac{kn-k+2}{2^n}.
\end{equation}
We claim that $x_2=\frac{1}{4}$. Assume that $x_2\ne\frac{1}{4}$.
Then $y_2\ne \frac{1}{2}$. Since $f(y)\ne\frac{kn-k+2}{2^n}$
for any $y\in({\mathbb F_{q}}\cup V)\setminus \{\frac{1}{2}\}$,
by Theorem 2.4 (i), we get that
$$D_{n,k}(1,x_2)=\frac{(k-1-(k-2)y_2)y_2^n-(1+(k-2)y_2)(1-y_2)^n}{2y_2-1}=f(y_2)\ne \frac{kn-k+2}{2^n},$$
which contradicts to (\ref{c4}). Hence the claim is true, and so we have $x_1=x_2$ as required.

{\sc Case 2.} Both of $x_1$ and $x_2$ are not equal to $\frac{1}{4}$. Then
$y_1\ne \frac{1}{2}$ and $y_2\ne \frac{1}{2}$. Since $D_{n,k}(1,x_1)=D_{n,k}(1,x_2)$,
by Theorem 2.4 (i), one has
$$\frac{(k-1-(k-2)y_1)y_1^n-(1+(k-2)y_1)(1-y_1)^n}{2y_1-1}$$
$$=\frac{(k-1-(k-2)y_2)y_2^n-(1+(k-2)y_2)(1-y_2)^n}{2y_2-1},$$
which is equivalent to $f(y_1)=f(y_2)$. However, $f$ is a
$2$-to-$1$ mapping on $({\mathbb F_{q}}\cup V)\setminus \{\frac{1}{2}\}$, and $f(y_2)=f(1-y_2)$ by
the definition of $f$. It then follows that $y_1=y_2$ or $y_1=1-y_2$. Thus $x_1=x_2$ as desired.
Hence the sufficiency part is proved.

Now we prove the necessity part. Let $D_{n,k}(1,x)$ be a PP of ${\mathbb F_{q}}$.
Choose two elements $y_1, y_2\in({\mathbb F_{q}}\cup V)\setminus \{\frac{1}{2}\}$
such that $f(y_1)=f(y_2)$, that is,
\begin{align}\label{c5}
&\frac{(k-1-(k-2)y_1)y_1^n-(1+(k-2)y_1)(1-y_1)^n}{2y_1-1}\notag\\
=&\frac{(k-1-(k-2)y_2)y_2^n-(1+(k-2)y_2)(1-y_2)^n}{2y_2-1}.
\end{align}
Since $y_1, y_2\in({\mathbb F_{q}}\cup V)\setminus \{\frac{1}{2}\}$,
it follows from Lemma 2.10 that $y_1(1-y_1)\in{\mathbb F_{q}}$ and
$y_2(1-y_2)\in{\mathbb F_{q}}$. So by Theorem 2.4 (i), (\ref{c5}) implies that
$$D_{n,k}(1,y_1(1-y_1))=D_{n,k}(1,y_2(1-y_2)).$$
Thus $y_1(1-y_1)=y_2(1-y_2)$ since $D_{n,k}(1,x)$ is a PP of ${\mathbb F_{q}}$,
which infers that $y_1=y_2$ or $y_1=1-y_2$. Since $y_2\ne \frac{1}{2}$,
one has $y_2\ne 1-y_2$. Therefore $f$ is a
$2$-to-$1$ mapping on $({\mathbb F_{q}}\cup V)\setminus \{\frac{1}{2}\}$.

Now take $y'\in ({\mathbb F_{q}}\cup V)\setminus \{\frac{1}{2}\}$.
Then from Lemma 2.10 it follows that $y'(1-y')\in{\mathbb F_{q}}$ and
$$y'(1-y')\ne \frac{1}{2}\Big(1-\frac{1}{2}\Big).$$
Notice that $D_{n,k}(1,x)$ is a PP of ${\mathbb F_{q}}$. Hence one has
$$D_{n,k}(1,y'(1-y'))\ne D_{n,k}\Big(1,\frac{1}{2}\Big(1-\frac{1}{2}\Big)\Big).$$
But Theorem 2.4 (i) tells us that
$$
D_{n,k}\Big(1,\frac{1}{2}\Big(1-\frac{1}{2}\Big)\Big)=\frac{kn-k-2}{2^n}.
$$
Then by Theorem 2.4 (i) and noting that $y'\ne \frac{1}{2}$, we have
$$\frac{\big(k-1-(k-2)y'\big)y'^{n}-\big(1+(k-2)y'\big)(1-y')^{n}}{2y'-1},$$
which infers that $f(y')\ne \frac{kn-k-2}{2^n}$ for any
$y'\in ({\mathbb F_{q}}\cup V)\setminus \{\frac{1}{2}\}$.
So the necessity part is proved.

The proof of Theorem 2.11 is complete.
\end{proof}

\section{A necessary condition for $D_{n,k}(1,x)$ to be permutational
and an auxiliary polynomial}

In this section, we study a necessary condition on $n$ for $D_{n,k}(1,x)$
to be a PP of ${\mathbb F_{q}}$. In particular, if $k=3$, then it is easy to check that
$$D_{0,k}(1,0)=2-k, D_{n,k}(1,0)=1$$
for any $n\ge 1$ and
$$ D_{0,k}(1,1)=2-k, D_{1,k}(1,1)=1, D_{n+2,k}(1,1)=D_{n+1,k}(1,1)-D_{n,k}(1,1)$$
for $n\ge 0$
, then one can easily show that the sequences
$\{D_{n,k}(1,1)|n\in \mathbb{N}\}$
are periodic with the smallest positive periods 6. In fact, one has
$$
D_{n,k}(1,1)={\left\{
  \begin{array}{rl}
2-k,& {\rm if} \ n\equiv 0\pmod{6},\\
1,& {\rm if} \ n\equiv 1\pmod{6},\\
k-1,& {\rm if} \ n\equiv 2\pmod{6},\\
k-2,& {\rm if} \ n\equiv 3\pmod{6},\\
-1,& {\rm if} \ n\equiv 4\pmod{6},\\
1-k,& {\rm if} \ n\equiv 5\pmod{6}
\end{array}
\right.}
$$
\begin{thm}
Assume that $D_{n,k}(1,x)$ is a PP of ${\mathbb F_{q}}$ with $q=p^e$ and $p>3$. Then
$n\not\equiv 1\pmod{6}$.
\end{thm}
\begin{proof}
Let $D_{n,k}(1,x)$ be a PP of ${\mathbb F_{q}}$. Then $D_{n,k}(1,0)$ and $D_{n,k}(1,1)$
are distinct. Then by the above results, the desired result $n\not\equiv 1 \pmod{6}$ follows immediately.
\end{proof}

Let $n, k$ be nonnegative integers. We define the following
auxiliary polynomial $p_{n,k}(x)\in \mathbb{Z}[x]$ by
$$p_{n,k}(x):=k\sum_{j\ge 0}\binom{n}{2j+1}x^j-(k-2)\sum_{j\ge 0}\binom{n}{2j}x^j$$
for $n\ge 1$ and $p_{0,k}(x):=2^n(2-k)$.
Then we have the following relation between $D_{n,k}(1,x)$ and $p_{n,k}(x)$.
\begin{thm}
Let $p>3$ be a prime and $n\ge 0$ be an even integer. Then

{\rm (i).} One has
$$D_{n,k}(1,x)=\frac{1}{2^n}f_{n}(1-4x). \eqno(3.1)$$

{\rm (ii).} We have that $D_{n,k}(1,x)$ is a PP of ${\mathbb F_{q}}$
if and only if $p_{n,k}(x)$ is a PP of ${\mathbb F_{q}}$.
\end{thm}
\begin{proof} (i). Clearly, (3.1) follows from the definitions of $p_{0,k}(x)$
and $D_{0,k}(1,x)$ if $n=0$. Then we assume that $n\ge 1$ in what follows.

First, let $x\in{\mathbb F_{q}}\setminus\{\frac{1}{4}\}$. Then there exists
$y\in{\mathbb F_{q^2}}\setminus\{\frac{1}{2}\}$ such that $x=y(1-y)$. Let $u=2y-1$.
It then follows from Theorem 2.4 (i) that
\begin{align*}
&D_{n,k}(1,x)\\
&=D_{n,k}(1,y(1-y))\\
&=\frac{\big(k-1-(k-2)y\big)y^{n}-\big(1+(k-2)y\big)(1-y)^{n}}{2y-1}\\
&=\frac{\frac{-(k-2)u+k}{2}\big(\frac{u+1}{2}\big)^{n}-\frac{(k-2)u+k}{2}\big(\frac{1-u}{2}\big)^{n}}{u}\\
&=\frac{1}{2^{n+1}u}\Big(k\big((u+1)^n-(1-u)^n\big)-(k-2)u\big((u+1)^n+(1-u)^n\big)\Big)\\
&=\frac{1}{2^{n}}\Bigg ( k\sum_{j\ge 0}\binom{n}{2j+1}x^j-(k-2)\sum_{j\ge 0}\binom{n}{2j}u^{2j}\Bigg)\\
&=\frac{1}{2^n}p_{n,k}(u^2)\\
&=\frac{1}{2^n}p_{n,k}(1-4y(1-y))\\
&=\frac{1}{2^n}p_{n,k}(1-4x)
\end{align*}
as desired. So (3.1) holds in this case.

Consequently, we let $x=\frac{1}{4}$. Then by Theorem 2.4 (i),
we have
$$D_{n,k}\Big(1,\frac{1}{4}\Big)=\frac{kn-k+2}{2^n}.$$
On the other hand, we can easily check that $p_{n,k}(0)=kn-k+2$. Therefore
$$D_{n,k}\Big(1,\frac{1}{4}\Big)=\frac{1}{2^n}p_{n,k}(0)
=\frac{1}{2^n}p_{n,k}(0)\Big(1-4\times \frac{1}{4}\Big)$$
as one desires. So (3.1) is proved.

(ii).
Notice that $\frac{1}{2^n}\in{\mathbb F_{q}^*}$ and $1-4x$
is linear. So $D_{n,k}(1,x)$ is a PP of ${\mathbb F_{q}}$ if and only if
$p_{n,k}(x)$ is a PP of ${\mathbb F_{q}}$. This ends the proof of Theorem 3.2.
\end{proof}

\section{The first moment $\sum_{a\in {\mathbb F_{q}}}D_{n,k}(1,a)$}

In this section, we compute the first moment
$\sum_{a\in {\mathbb F_{q}}}D_{n,k}(1,a)$.
By Proposition 2.6, one has
\begin{align}
\sum_{n=0}^{\infty}D_{n,k}(1,x)t^n&=\frac{(k-1)t-k+2}{1-t+xt^2}
=\frac{(k-1)t-k+2}{1-t}\frac{1}{1-\frac{t^2}{t-1}x} \notag\\
&=\frac{(k-1)t-k+2}{1-t}\Big(1+\sum_{m=1}^{q-1}\sum_{\ell=0}^{\infty }
\bigg(\frac{t^2}{t-1}\bigg)^{m+\ell (q-1)}x^{m+\ell (q-1)}\Big) \notag\\
&\equiv\frac{2t-1}{1-t}\Big(1+\sum_{m=1}^{q-1}\sum_{\ell=0}^{\infty }
\bigg(\frac{t^2}{t-1}\bigg)^{m+\ell (q-1)}x^{m}\Big)\pmod{x^q-x} \notag\\
&=\frac{(k-1)t-k+2}{1-t}\Big(1+\sum_{m=1}^{q-1}
\frac{(\frac{t^2}{t-1})^m}{1-(\frac{t^2}{t-1})^{q-1}}x^m\Big) \notag\\
&=\frac{(k-1)t-k+2}{1-t}\Big(1+\sum_{m=1}^{q-1}
\frac{(t-1)^{q-1-k}t^{2m}}{(t-1)^{q-1}-t^{2(q-1)}}x^m\Big).\label{c6}
\end{align}
Moreover, by Theorem 2.4 (ii), it follows that for any
$x\in{\mathbb F_{q}}\setminus\{\frac{1}{4}\}$, one has
$$D_{n_1,k}(1,x)=D_{n_2,k}(1,x)$$
when $n_1\equiv n_2\pmod{q^2-1}$. Thus if $x\ne \frac{1}{4}$, one has
\begin{align}
\sum_{n=0}^\infty D_{n,k}(1,x)t^n&=1+\sum_{n=1}^{q^2-1}
\sum_{\ell=0}^{\infty }D_{n+\ell(q^2-1),k}(1,x)t^{n+\ell(q^2-1)}\notag\\
&=1+\sum_{n=1}^{q^2-1}D_{n,k}(1,x)\sum_{\ell=0}^{\infty }t^{n+\ell(q^2-1)}\notag\\
&=1+\frac{1}{1-t^{q^2-1}}\sum_{n=1}^{q^2-1}D_{n,k}(1,x)t^{n}.\label{c7}
\end{align}
Then (\ref{c6}) together with (\ref{c7}) gives that
for any $x\ne \frac{1}{4}$, we have
\begin{align}
&\sum_{n=1}^{q^2-1}D_{n,k}(1,x)t^n\notag\\
=&\Big(\sum_{n=0}^{\infty}D_{n,k}(1,x)t^n-1\Big)(1-t^{q^2-1})\notag\\
\equiv & \Big(\frac{(k-1)t-k+2}{1-t}-1\Big)(1-t^{q^2-1})\notag\\
&+\frac{(1-t^{q^2-1})((k-1)t-k+2)}{1-t}\sum_{m=1}^{q-1}
\frac{(t-1)^{q-1-m}t^{2m}}{(t-1)^{q-1}-t^{2(q-1)}}x^m\pmod{x^q-x}\notag\\
=&\frac{(kt+1-k)(1-t^{q^2-1})}{1-t}+h(t)\sum_{m=1}^{q-1}
(t-1)^{q-1-m}t^{2m}x^m,\label{c8}
\end{align}
where

$$h(t):=\frac{(t^{q^2-1}-1)((k-1)t-k+2)}{(t-1)^q-(t-1)t^{2(q-1)}}.$$
\begin{lem}\cite{[LN]}
Let $u_0,u_1,\cdots,u_{q-1}$ be the list of the all elements of
${\mathbb F_{q}}$. Then
$$
\sum_{i=0}^{q-1}u_i^k={\left\{
  \begin{array}{rl}
0,& {\rm if} \ 0\le k\le q-2,\\
-1,& {\rm if} \ k=q-1.
\end{array}
\right.}
$$
\end{lem}

Now by Theorem 2.4 (i), Lemma 4.1 and (\ref{c8}), we derive that
\begin{align}
&\sum_{n=1}^{q^2-1}\sum_{a\in{\mathbb F_{q}}}D_{n,k}(1,a)t^n\notag\\
=&\sum_{n=1}^{q^2-1}D_{n,k}\Big(1,\frac{1}{4}\Big)t^n+\sum_{n=1}^{q^2-1}
\sum_{a\in{\mathbb F_{q}}\setminus\{\frac{1}{4}\}}D_{n,k}(1,a)t^n\notag\\
=&\sum_{n=1}^{q^2-1}\frac{kn-k+2}{2^n}t^n+\sum_{a\in{\mathbb F_{q}}
\setminus\{\frac{1}{4}\}}\frac{(kt+1-k)(1-t^{q^2-1})}{1-t}+h(t)\sum_{m=1}
^{q-1}(t-1)^{q-1-m}t^{2m}\sum_{a\in{\mathbb F_{q}}\setminus\{\frac{1}{4}\}}a^m\notag\\
=&\sum_{n=1}^{q^2-1}\frac{kn-k+2}{2^n}t^n+(q-1)\frac{(kt+1-k)(1-t^{q^2-1})}{1-t}
+h(t)\sum_{m=1}^{q-1}(t-1)^{q-1-m}t^{2m}\sum_{a\in{\mathbb F_{q}}}a^m\notag\\
&\ \ \ -h(t)\sum_{m=1}^{q-1}(t-1)^{q-1-m}t^{2m}\Big(\frac{1}{4}\Big)^m\notag\\
=&\sum_{n=1}^{q^2-1}\frac{kn-k+2}{2^n}t^n-\frac{(kt+1-k)(1-t^{q^2-1})}{1-t}-h(t)t^{2(q-1)}
-h(t)\sum_{m=1}^{q-1}(t-1)^{q-1-m}t^{2m}\Big(\frac{1}{4}\Big)^m.\label{c9}
\end{align}

Since $(t-1)^q=t^q-1$ and $q$ is odd, one has
\begin{align}
h(t)=&\frac{(t^{q^2-1}-1)(2t-1)}{(t-1)^q-(t-1)t^{2(q-1)}}\notag\\
=&\frac{(t^{q^2-1}-1)(2t-1)}{(1-t^{q-1})(t^q-t^{q-1}-1)} \notag\\
=&\frac{(t^{q^2}-t)(2t-1)}{(t-t^{q})(t^q-t^{q-1}-1)}\notag\\
=&\frac{(t^{q}-t)^q+t^q-t}{t-t^{q}}\cdot \frac{2t-1}{t^q-t^{q-1}-1}\notag\\
=&\frac{(-1-(t-t^q)^{q-1})(2t-1)}{t^q-t^{q-1}-1}\notag\notag\\
=&\frac{(2t-1)\sum_{i=0}^{q^2-q}b_it^i}{t^q-t^{q-1}-1},\label{c10}
\end{align}
where
$$\sum_{i=0}^{q^2-q}b_it^i:=-1-(t-t^q)^{q-1}.$$
Then by the binomial theorem applied to $(t-t^q)^{q-1}$,
we can derive the following expression for the coefficient $b_i$.
\begin{prop}
For each integer $i$ with $0\le i\le q^2-q$, write $i=\alpha+\beta q$
with $\alpha $ and $\beta$ being integers such that
$0\le \alpha,\beta \le q-1$. Then
$$b_i={\left\{
\begin{array}{ll}
(-1)^{\beta +1}\binom{q-1}{\beta},& {\it if} \ \alpha +\beta =q-1,\\
-1,& {\it if} \ \alpha =\beta =0,\\
0,& {\it otherwise}.
\end{array}
\right.}$$
\end{prop}
For convenience, let
$$a_n:=\sum_{a\in{\mathbb F_{q}}}D_{n,k}(1,a).$$
Then by (\ref{c9}) and (\ref{c10}), we arrive at
\begin{align}
&\sum_{n=1}^{q^2-1}\Big(a_n-\frac{kn-k+2}{2^n}\Big)t^n\\
=&-\frac{(kt+1-k)(1-t^{q^2-1})}{1-t}-\frac{(2t-1)\sum_{i=0}^{q^2-q}b_it^i}{t^q-t^{q-1}-1}
\Big(t^{2(q-1)}+\sum_{m=1}^{q-1}(t-1)^{q-1-m}t^{2m}\Big(\frac{1}{4}\Big)^m\Big),
\end{align}
which implies that
\begin{align}
&(t^q-t^{q-1}-1)\sum_{n=1}^{q^2-1}\Big(a_n-\frac{3n-1}{2^n}\Big)t^n\notag\\
=&-(t^q-t^{q-1}-1)(kt+1-k)\sum_{i=0}^{q^2-2}t^i-(2t-1)\Big(t^{2(q-1)}+\sum_{k=1}^{q-1}
(t-1)^{q-1-k}t^{2k}\Big(\frac{1}{4}\Big)^k\Big)\sum_{i=0}^{q^2-q}b_it^i. \label{c11}
\end{align}

Let
$$\sum_{i=1}^{q^2+q-1}c_it^i$$
denote the right-hand side of (\ref{c11}) and let
$$d_n:=a_n-\frac{kn-k+2}{2^n}$$
for each integer $n$ with $1\le n\le q^2-1$. Then (\ref{c11}) can be reduced to
\begin{equation}\label{c12}
(t^q-t^{q-1}-1)\sum_{n=1}^{q^2-1}d_nt^n=\sum_{i=1}^{q^2+q-1}c_it^i.
\end{equation}
Then by comparing the coefficient of $t^i$ with $1\le i\le q^2+q-1$
of the both sides in (\ref{c12}), we derive the following relations:
$${\left\{
  \begin{array}{ll}
c_j=-d_j,& {\rm if} \ 1\le j\le q-1,\\
c_q=-d_1-d_q,& \\
c_{q+j}=d_j-d_{j+1}-d_{q+j},& {\rm if} \ 1\le j\le q^2-q-1,\\
c_{q^2+j}=d_{q^2-q+j}-d_{q^2-q+j+1},& {\rm if} \ 0\le j\le q-2,\\
c_{q^2+q-1}=d_{q^2-1},&
\end{array}
\right.}$$
from which we can deduce that
\begin{align}\label{c13}
{\left\{
  \begin{array}{ll}
d_j=-c_j,& {\rm if} \ 1\le j\le q-1,\\
d_q=c_1-c_q,& \\
d_{\ell q+j}=d_{(\ell -1)q+j}-d_{(\ell-1)q+j+1}-c_{\ell q+j},
& {\rm if} \ 1\le \ell\le q-2\ \ {\rm and}\ \ 1\le j\le q-1,\\
d_{\ell q}=d_{(\ell -1)q}-d_{(\ell-1)q+1}-c_{\ell q},& {\rm if} \ 2\le \ell\le q-2,\\
d_{q^2-q+j}=\sum_{i=j}^{q-1}c_{q^2+i},& {\rm if} \ 0\le j\le q-1.
\end{array}
\right.}
\end{align}

Finally, (\ref{c13}) together with the following identity
$$\sum_{a\in {\mathbb F_{q}}}D_{n,k}(1,a)=d_n+\frac{kn-k+2}{2^n}$$
shows that the last main result of this paper is true:
\begin{thm}
Let $c_i$ be the coefficient of $t^i$ in the right-hand side of (\ref{c11})
with $i$ being an integer such that $1\le i\le q^2+q-1$. Then we have
\begin{align*}
&\sum_{a\in {\mathbb F_{q}}}D_{j,k}(1,a)=-c_j+\frac{kj-k+2}{2^j}\ \ if\ \ 1\le j\le q-1,\\
&\sum_{a\in {\mathbb F_{q}}}D_{q,k}(1,a)=c_1-c_q-\frac{k-2}{2},\\
&\sum_{a\in {\mathbb F_{q}}}D_{\ell q+j,k}(1,a)=\sum_{a\in {\mathbb F_{q}}}D_{(\ell-1)q+j,k}(1,a)-
\sum_{a\in {\mathbb F_{q}}}D_{(\ell-1)q+j+1,k}(1,a)-c_{\ell q+j}+\frac{k}{2^{\ell +j}}\\
&\ \ \ \ if\ \ 1\le \ell\le q-2 \ {\it and} \ 1\le j\le q-1,\\
&\sum_{a\in {\mathbb F_{q}}}D_{\ell q,k}(1,a)=\sum_{a\in {\mathbb F_{q}}}D_{(\ell-1)q,k}(1,a)-
\sum_{a\in {\mathbb F_{q}}}D_{(\ell-1)q+1,k}(1,a)-c_{\ell q}+\frac{k}{2^{\ell}}\ \ if\ \ 2\le \ell\le q-2
\end{align*}
and
\begin{align*}
&\sum_{a\in {\mathbb F_{q}}}D_{q^2-q+j,k}(1,a)
=\sum_{i=j}^{q-1}c_{q^2+i}+\frac{kj-k+2}{2^j}\ \ if\ \ 0\le j\le q-1.
\end{align*}
\end{thm}
  
\begin{center}
{\bf Acknowledgment}
\end{center}
The author would like to thank Professor Shaofang Hong for pointing out
to him in early of this April that the arguments for the second kind case
and for the fourth kind case presented in \cite{[HQZ]} and \cite{[CHQ]},
respectively, work also for the general $(k+1)$-th kind case.


\begin{thebibliography}{99}
\bibitem{[CHQ]} K. Cheng, S. Hong and X. Qin, Reversed Dickson polynomials
of the fourth kind over finite fields, arXiv:1604.04557.
\bibitem{[Coh]} S.D. Cohen, Dickson polynomials of the second kind that
are permutations, Canad. J. Math. 46 (1994), 225-238.
\bibitem{[HQZ]} S. Hong, X. Qin and W. Zhao, Necessary conditions for
reversed Dickson polynomials of the second kind to be permutational,
Finite Fields Appl. 37 (2016), 54-71.
\bibitem{[HL]} X. Hou and T. Ly, Necessary conditions for reversed Dickson
polynomials to be permutational, Finite Fields Appl. 16 (2010), 436-448.
\bibitem{[HMSY]} X. Hou, G.L. Mullen, J.A. Sellers and J.L. Yucas,
Reversed Dickson polynomials over finite fields, Finite Fields Appl.
15 (2009), 748-773.
\bibitem{[LN]} R. Lidl and H. Niederreiter, Finite Fields, second ed.,
Encyclopedia of Mathematics and its Applications, vol.20, Cambridge
University Press, Cambridge, 1997.
\bibitem{[WY]} Q. Wang and J.L. Yucas, Dickson polynomials over
finite fields, Finite Fields Appl. 18 (2012), 814-831.
\end{thebibliography}
\end{document}